\documentclass{amsart}

\usepackage{graphicx}
\usepackage{color}

\newtheorem{theorem}{Theorem}[section]

\theoremstyle{definition}
\newtheorem{definition}[theorem]{Definition}

\newtheorem{proposition}[theorem]{Proposition}

\newcommand{\cosg}{\mathrm{cosg}}
\newcommand{\sing}{\mathrm{sing}}
\newcommand{\vecpe}{\times_{e}}

\theoremstyle{remark}
\newtheorem{remark}[theorem]{Remark}

\numberwithin{equation}{section}

\begin{document}

\title[RM frames and spherical curves in isotropic spaces]{Rotation minimizing frames and spherical curves in simply isotropic and pseudo-isotropic 3-spaces}


\author[Luiz C. B. da Silva]{Luiz C. B. da Silva}
\address{Departamento de Matem\'atica\\
Universidade Federal de Pernambuco\\
50670-901, Recife, Pernambuco, Brazil}
\curraddr{Department of Physics of Complex Systems, Weizmann Institute of Science, Rehovot 7610001, Israel}
\email{luiz.da-silva@weizmann.ac.il}

\thanks{This work was financially supported by Conselho Nacional de Desenvolvimento Cient\'ifico e Tecnol\'ogico - CNPq (Brazilian agency).
}


\date{\today}


\begin{abstract}
In this work, we are interested in the differential geometry of curves in the simply isotropic and pseudo-isotropic 3-spaces, which are examples of Cayley-Klein geometries whose absolute figure is given by a plane at infinity and a degenerate quadric. Motivated by the success of rotation minimizing (RM) frames in Euclidean and Lorentzian geometries, here we show how to build RM frames in isotropic geometries and apply them in the study of isotropic spherical curves. Indeed, through a convenient manipulation of osculating spheres described in terms of RM frames, we show that it is possible to characterize spherical curves via a linear equation involving the curvatures that dictate the RM frame motion. For the case of pseudo-isotropic space, we also discuss on the distinct choices for the absolute figure in the framework of a Cayley-Klein geometry and prove that they are all equivalent approaches through the use of Lorentz numbers (a complex-like system where the square of the imaginary unit is $+1$). Finally, we also show the possibility of obtaining an isotropic RM frame by rotation of the Frenet frame through the use of Galilean trigonometric functions and dual numbers (a complex-like system where the square of the imaginary unit vanishes).
\end{abstract}

\maketitle

\section{Introduction}

The three dimensional ($3d$) simply isotropic $\mathbb{I}^3$ and pseudo-isotropic $\mathbb{I}_{\mathrm{p}}^3$ spaces are examples of $3d$ Cayley-Klein (CK) geometries \cite{Giering1982,Sulanke2006,Sachs1987,StruveJG2010}, which is basically the study of those properties in projective space $\mathbb{P}^3$ that preserves a certain configuration, the so-called \emph{absolute figure}. Indeed, following Klein ``Erlanger Program'' \cite{birkhoff1988felix,klein1893vergleichende}, a CK geometry is the study of the geometry invariant by the action of the subgroup of projectivities that fix the absolute figure: e.g., Euclidean (Minkowski) space $\mathbb{E}^3$ ($\mathbb{E}_1^3$) can be modeled through an absolute figure given in homogeneous coordinates by a plane at infinity, usually identified with $x_0=0$, and a non-degenerate quadric of index zero (index one) usually identified with $x_0^2+\cdots+x_3^2=0$ ($x_0^2+x_1^2+x_2^2-x_3^2=0$, respectively) \cite{Giering1982}. In our case, i.e., isotropic geometries, the absolute figure is given by a plane at infinity and a degenerate quadric of index 0 or 1: $x_0^2+x_1^2+\delta\,x_2^2=0$, with $\delta=1$ for the simply isotropic figure and $\delta=-1$ for the pseudo-isotropic one.

Recently, isotropic geometry has been seen a renewed interest from both pure and applied viewpoints (a quite comprehensive and historical account before the 1990's can be found in \cite{Sachs1990}). We may mention investigations of special classes of curves \cite{YoonJG2017} and surfaces \cite{AydinJG2015,AydinArXiv2016,KaracanTJM2017,SipusPMH2014}, while applications may range from  economics \cite{AydinTJM2016,chenKJM2014} and elasticity \cite{pottmann2009laguerre} to image processing and shape interrogation \cite{koenderink2002image,PottmannCAGD1994}. Another stimulus may come from the problem of characterizing curves on level set surfaces $\Sigma=F^{-1}(c)$. In fact,  by introducing a metric induced by $\mbox{Hess}\,F$ \cite{daSilvaJG2017} one may be led to the study of an isotropic geometry, since the Hessian may fail to be non-degenerate: e.g., for  $\mbox{Hess}\,F=\mbox{diag}(1,\delta,0)$, the metric $\langle \mbox{Hess}\,F\,\cdot,\cdot\rangle$ leads to the geometry of simply isotropic space  if $\delta=1$, pseudo-isotropic space  if $\delta=-1$, and doubly isotropic space  if $\delta=0$ (see \cite{Sachs1990,SipusGM1998,StrubeckerSOA1941,VoglerGMB1989}, \cite{AydinArXiv2016}, and \cite{BraunerCrelle1967} for an account of these geometries, respectively).

Motived by the success of Rotation Minimizing (RM) frames in the study of spherical curves, here we develop the fundamentals of RM frames in isotropic spaces, which in combination with an adequate manipulation of osculating spheres allow us to prove that spherical curves can be characterized through a linear equation involving the coefficients that dictate the frame motion, as also happens in Euclidean $\mathbb{E}^3$ \cite{BishopMonthly}, Lorentzian $\mathbb{E}_1^3$ \cite{daSilvaJG2017,OzdemirMJMS2008}, and in Riemannian spaces \cite{daSilvaMJM2018,Etayo2016}(\footnote{The characterization of isotropic spherical curves via a Frenet frame is made through a differential equation involving curvature and torsion \cite{StrubeckerSOA1941}, see also Eq. (7.36) in \cite{Sachs1990}, p. 128.}). In addition, for the case of pseudo-isotropic space we discuss the construction of spheres, moving frames along curves, and pseudo-isotropic spherical indicatrix. We also discuss on the distinct approaches to the study of pseudo-isotropic space as a CK geometry, in which we are able to prove that the available choices are all equivalent with the help of the so-called Lorentz numbers \cite{BirmanMonthly1984,Yaglom1979}. Finally, we also show how to relate RM and Frenet frames via isotropic rotations, in both $\mathbb{I}^3$ and $\mathbb{I}_{\mathrm{p}}^3$, by using  Galilean trigonometric functions \cite{Yaglom1979} and dual numbers \cite{Sachs1987,Yaglom1979}.

The remaining of this work is divided as follows. In section 2, we review the concept of RM frames and spherical curves in Euclidean space. In section 3, we introduce some terminology related to simply isotropic space and, in section 4, we discuss how to introduce moving frames along simply isotropic curves. In section 5, we then study simply isotropic spheres and the characterization of spherical curves. In section 6, we turn our attention to the pseudo-isotropic space. In section 7 and 8, we study pseudo-isotropic spheres and moving frames along pseudo-isotropic curves, respectively. In section 9, we  characterize pseudo-isotropic spherical curves. Finally, the Appendix contains a short account of the rings of dual and Lorentz numbers. 

\begin{remark}
Despite the risk of making this paper longer than what would be strictly necessary, here we will try to be as self-contained as possible, since some of the most comprehensive and elementary references in isotropic geometry, such as \cite{Giering1982,Sachs1987,Sachs1990,StrubeckerCrelle1938,StrubeckerSOA1941}, are not available in English. We hope this will make the concepts from isotropic geometry more accessible to a broader audience.
\end{remark}

\section{Preliminaries: rotation minimizing frames and spherical curves in Euclidean space}

Let $\mathbb{E}^3$ be the $3d$ Euclidean space, i.e., $\mathbb{R}^3$ equipped with the standard Euclidean metric $\langle\cdot,\cdot\rangle$. The usual way to introduce a moving frame along curves is by means of the Frenet frame \cite{Kreyszig1991,Kuhnel2010}. However, there are other possibilities as well. Indeed, by introducing the notion of a rotation minimizing vector field, Bishop considered an orthonormal adapted moving frame $\{\mathbf{t},\mathbf{n}_1,\mathbf{n}_2\}$, where $\mathbf{t}$ is the unit tangent, whose equations of motion are \cite{BishopMonthly} 
\begin{equation}
\mathbf{t}'(s)=\kappa_1(s)\,\mathbf{n}_1(s)+\kappa_2(s)\,\mathbf{n}_2(s)\mbox{ and }\mathbf{n}_i'(s)=-\kappa_i(s)\,\mathbf{t}(s).\label{eq::BishopEqs}
\end{equation}

The basic idea here is that $\mathbf{n}_i$ rotates
only the necessary amount to remain normal to the tangent (then justifying the terminology). In addition,  $\kappa_1$ and $\kappa_2$ relate with the curvature $\kappa$ and torsion $\tau$ according to \cite{BishopMonthly}
\begin{equation}
\left\{
\begin{array}{c}
\kappa_1 = \kappa\cos\theta\\
\kappa_2 = \kappa\sin\theta\\
\end{array}
\right.\mbox{ and }\,\theta'= \tau.
\end{equation}

Notice that RM  frames are not uniquely defined, any rotation of $\mathbf{n}_i$ on the normal plane still gives a new RM vector field, i.e., there is an ambiguity associated with the group $SO(2)$ acting on the normal planes. So, an RM frame is defined up to an additive constant\footnote{Despite of this, the prescription of $\kappa_1,\kappa_2$ still determines a curve up to rigid motions and, in addition, RM frames can be globally defined even if the curvature $\kappa$ has a zero \cite{BishopMonthly}.}. Finally, of great interest to us, is the
\begin{theorem}[\cite{BishopMonthly,daSilvaJG2017}]
A regular curve in Euclidean or Lorentz-Minkowski spaces lies on a sphere if and only if its normal development, i.e., the curve $(\kappa_1(s),\kappa_2(s))$, lies on a line not passing through the origin. 
In addition, straight lines passing through the origin characterize plane curves which are not spherical.
\label{theo:BishopCharacSpherericalCurves} 
\end{theorem}

Here we furnish a proof of the above result by using osculating spheres, whose parametrization using  RM frames may be written as
\begin{equation}
P_S(s_0)=\alpha(s_0)+\beta_0\mathbf{t}(s_0)+\beta_1 \mathbf{n}_1(s_0)+\beta_2\mathbf{n}_2(s_0).
\end{equation}
Now, defining  $g(s)=\langle P_S-\alpha(s),P_S-\alpha(s)\rangle-r^2$, we have
\begin{eqnarray}
g ' & = & -2\langle P_S-\alpha,\mathbf{t}\rangle = -2\beta_0\,,\\
g'' & = & 2\langle\mathbf{t},\mathbf{t}\rangle-2\langle P_S-\alpha,\kappa_1\mathbf{n}_1+\kappa_2\mathbf{n}_2\rangle = -2(-1+\kappa_1\beta_1+\kappa_2\beta_2)\,,\\
g''' & = & -2\langle P_S-\alpha,\sum_i(\kappa_i'\mathbf{n}_i-\kappa_i^2\mathbf{t})\,\rangle = -2\sum_i(\kappa_i'\beta_i-\kappa_i^2\beta_0)\,.
\end{eqnarray}
Imposing an order 3 contact leads to $g'(s_0)=g''(s_0)=g'''(s_0)=0$ and then
\begin{equation}
\beta_0 = 0,\,\kappa_1(s_0)\beta
_1+\kappa_2(s_0)\beta_2-1=0,\,\mbox{ and }\kappa_1'(s_0)\beta_1+\kappa_2'(s_0)\beta_2=0.\label{eq::coefOfOscSphUsingRMF}
\end{equation}
Thus, the coefficients $\beta_0$, $\beta_1$, and $\beta_2$ as functions of $s_0$ are
\begin{equation}
\beta_0=0,\,\beta_1=\frac{\kappa_2'}{\kappa_1\kappa_2'-\kappa_1'\kappa_2}=\frac{\kappa_2'}{\tau\kappa^2},\,\mbox{ and }\beta_2=-\frac{\kappa_1'}{\kappa_1\kappa_2'-\kappa_1'\kappa_2}=-\frac{\kappa_1'}{\tau\kappa^2}\,,
\end{equation}
where in the equalities above we also used the relation between $(\kappa_1,\kappa_2)$ and $(\kappa,\tau)$.
\newline
\newline
\textit{Proof of Theorem \ref{theo:BishopCharacSpherericalCurves} for $C^4$ curves.} The derivative of the osculating center gives
\begin{equation}
P_S' = \frac{{\rm d}}{{\rm d}s}\left(\alpha+\frac{\kappa_2'}{\tau\kappa^2}\mathbf{n}_1-\frac{\kappa_1'}{\tau\kappa^2}\mathbf{n}_2\right)\nonumber\\
=\left(\frac{{\rm d}}{{\rm d}s}\frac{\kappa_2'}{\tau\kappa^2}\right)\mathbf{n}_1-\left(\frac{{\rm d}}{{\rm d}s}\frac{\kappa_1'}{\tau\kappa^2}\right)\mathbf{n}_2\,.
\end{equation}
From the linear independence of $\{\mathbf{n}_1,\mathbf{n}_2\}$ we conclude that $\alpha$ is spherical, i.e., $P'_S=0$, if and only if $\beta_1$ and $\beta_2$ are constants. From Eq. (\ref{eq::coefOfOscSphUsingRMF}), this is equivalent to say that the normal development lies on a line not passing through the origin.
\qed

\begin{remark}
The proof above has some weaknesses when compared with that of Bishop in  $\mathbb{E}^3$ \cite{BishopMonthly}. The use of osculating spheres demands that the curve must be $C^4$ and also that $\tau\not=0$, while in Bishop's approach one needs just a $C^2$ condition and no restriction on the torsion: $C^2$ is enough to have $\mathbf{t}$ and $\kappa_i$, while we need a $C^4$ to have $\kappa_i''$. However, this approach will prove to be very useful in the following due to the lack of good orthogonality properties in isotropic spaces.
\end{remark}

In the following, we shall extend this formalism in order to present a way of
building RM frames along curves in both simply isotropic $\mathbb{I}^3$ and pseudo-isotropic $\mathbb{I}_{\mathrm{p}}^3$ 3-spaces and then apply them to furnish a unified
approach to the characterization of isotropic spherical curves. In addition, by employing dual numbers and Galilean trigonometric functions, we will also show how to relate (i) a Frenet frame to an RM frame and (ii) an RM frame with another RM frame through isotropic rotations.

\section{Differential geometry in simply isotropic space}

In the spirit of Klein's Erlangen Program, simply isotropic geometry is the study of the properties invariant by the action of the 6-parameter group $\mathcal{B}_6$ \cite{Sachs1990}
\begin{equation}
\left\{
\begin{array}{ccc}
\bar{x} & = & a + x\,\cos\phi-y\,\sin\phi\\ 
\bar{y} & = & b + x\,\sin\phi+y\,\cos\phi\\
\bar{z} & = & c_0 + c_1x+c_2y+z\\
\end{array}
\right.,\,a,\,b,\,c_i,\,\phi\in\mathbb{R}.\label{eq::IsoGroupB6}
\end{equation}
In other words, $\mathcal{B}_6$ is our group of rigid motions. Notice that on the $xy$-plane this geometry looks exactly like the plane Euclidean geometry $\mathbb{E}^2$. The projection of a vector $\mathbf{u}=(u_1,u_2,u_3)$ on the $xy$-plane is the \emph{top view} of $\mathbf{u}$ and we shall denote it by $\tilde{\mathbf{u}}=(u_1,u_2,0)$. The top view concept plays a fundamental role in the simply isotropic space $\mathbb{I}^3$, since the $z$-direction is preserved by the action of $\mathcal{B}_6$. A line with this direction is called an \emph{isotropic line} and a plane that contains an isotropic line is said to be an \emph{isotropic plane}.

One may introduce a \emph{simply isotropic inner product} between two vectors $\mathbf{u},\,\mathbf{v}$ as
\begin{equation}
\langle\,(u_1,u_2,u_3),(v_1,v_2,v_3)\,\rangle_{z} = u_1v_1+u_2v_2\,,
\end{equation}
from which we define a \emph{simply isotropic distance} as\footnote{The index $z$ here emphasizes that $z$ is the isotropic (degenerate) direction. Note, in addition, that $\langle\cdot,\cdot\rangle_z$ induces a semi-distance in $\mathbb{R}^3$, since points in an isotropic line have zero distance.}: 
\begin{equation}
\mathrm{d}_z(A,B)=\sqrt{\langle B-A,B-A\rangle_z}.
\end{equation}

 The inner product and distance above are just the plane Euclidean counterparts of the top views $\tilde{\mathbf{u}}$ and $\tilde{\mathbf{v}}$. In addition, since the isotropic metric is degenerate, the distance from $(u_1,u_2,u_3)$ to $(u_1,u_2,v_3)$ is zero ($\tilde{\mathbf{u}}=\tilde{\mathbf{v}}$). In such cases, one may define a codistance by $\mathrm{cd}_z(A,B)=\vert b_3-a_3\vert$, which is then preserved by $\mathcal{B}_6$. (It would be interesting to mention that $\mathbb{I}^3$ is not isotropic from a ``physicist viewpoint", since the $z$-direction is preserved by rigid motions and then gives rise to a certain anisotropy. In any case, this is an established nomenclature and we keep it here.)

\section{Moving frames along curves in simply isotropic space}

A regular curve $\alpha:I\to \mathbb{I}^3$, i.e., $\alpha'\not=0$, is parameterized by arc-length $s$ when $\Vert\alpha'(s)\Vert_z\stackrel{\mathrm{def}}{=}\Vert\tilde{\alpha}'(s)\Vert=1$. In the following, we shall assume it for all curves (in particular, this excludes isotropic velocity vectors). In addition, a point $\alpha(s_0)$ in which $\{\alpha'(s_0),\alpha''(s_0)\}$ is linearly dependent is an \emph{inflection point} and a regular unit speed curve $\alpha(s)=(x(s),y(s),z(s))$  with no inflection point is called an \emph{admissible curve} if $x'y''-x''y'\not=0$ (this condition implies that the osculating planes, i.e., the planes that have a contact of order 2 with the reference curve\footnote{For a level set surface $\Sigma=G^{-1}(c)$, a contact of order $k$ with $\alpha$ at $\alpha(s_0)$ is equivalent to say that $\beta^{(i)}(s_0)=0$ ($1\leq i\leq k$), where $\beta=G\circ\alpha$ and $c=\beta(s_0)=\alpha(s_0)$ \cite{Kreyszig1991}.}, can not be isotropic. Moreover, the only curves with $x'y''-x''y'\equiv0$ are precisely the isotropic lines \cite{Sachs1990}).

\subsection{Simply isotropic Frenet frame}

The (isotropic) unit tangent $\mathbf{t}$, principal normal $\mathbf{n}$, and curvature function $\kappa$ are defined as usual
\begin{equation}
\mathbf{t}(s)=\alpha'(s),\,\mathbf{n}(s)=\frac{\mathbf{t}'(s)}{\kappa(s)},\,\,\mathrm{ and }\,\,\kappa(s)=\Vert\mathbf{t}'(s)\Vert_z=\Vert\tilde{\mathbf{t}}'(s)\Vert.
\end{equation}
As usually happens in isotropic geometry, the curvature $\kappa$ is just the plane curvature function of its top view $\tilde{\alpha}$: $\kappa(s)=(x'y''-x''y')(s)$. To complete the trihedron, we define the binormal as the (co)unit vector $\mathbf{b}=(0,0,1)$ in the isotropic direction. The frame  $\{\mathbf{t},\mathbf{n},\mathbf{b}\}$ is linearly independent, $\det(\mathbf{t},\mathbf{n},\mathbf{b})=\frac{1}{\kappa}(x'y''-x''y')=1$, and the Frenet equations corresponding to the isotropic Frenet frame are
\begin{equation}
\frac{\mathrm{d}}{\mathrm{d}s}\left(
\begin{array}{c}
\mathbf{t}\\
\mathbf{n}\\
\mathbf{b}\\
\end{array}
\right)=\left(
\begin{array}{ccr}
0\,\, & \kappa &\,\, 0\\
-\kappa\,\, & 0 & \,\,\tau\\
0\,\, & 0 &\,\, 0 \\
\end{array}
\right)\left(
\begin{array}{c}
\mathbf{t}\\
\mathbf{n}\\
\mathbf{b}\\
\end{array}
\right),
\end{equation}
where $\tau$ is the (isotropic) torsion \cite{Sachs1990}, p. 110:
\begin{equation}
\tau=\frac{\det(\alpha',\alpha'',\alpha''')}{\det(\tilde{\alpha}',\tilde{\alpha}'')}\,;\,\,\kappa=\frac{\det(\tilde{\alpha}',\tilde{\alpha}'')}{\sqrt{\langle\alpha',\alpha'\rangle_z}\,^3}\,.
\end{equation}

The above expressions for $\tau$ and $\kappa$ are also valid for any generic regular parameterization of $\alpha$. But, contrary to the Euclidean space $\mathbb{E}^3$, we can not define the torsion through the derivative of the binormal vector. However, remembering that the idea behind the definition of torsion in $\mathbb{E}^3$ is that of measuring the variation of the osculating plane, we may ask if $\tau\equiv0$ still characterizes plane curves in $\mathbb{I}^3$. It can be shown that the isotropic torsion is directly associated with the velocity of variation of the osculating plane, see \cite{Sachs1990}, pp. 112-113, and that an admissible curve lies on a non-isotropic plane if and only if $\tau$ vanishes. Observe, in addition, that contrary to the isotropic curvature, the torsion is not defined as the torsion of the top view, since this would result in $\tau=0$. The isotropic torsion is an intermediate concept depending on its top view behavior and on how much the curve leaves the plane spanned by $\alpha'$ and $\alpha''$.

\subsection{Rotation minimizing frames in simply isotropic space}

Let $\alpha:I\to\mathbb{I}^3$ be an admissible curve. A normal vector field $\mathbf{v}$ is a \emph{simply isotropic RM vector field} if $\mathbf{v}'=\mu\, \mathbf{t}$, for some function $\mu$. We can easily see that the binormal $\mathbf{b}$ is an RM vector field, $\mathbf{b}'=0$, and that, except for plane curves, the principal normal fails to be RM: $\mathbf{n}'=-\kappa\mathbf{t}+\tau\mathbf{b}$. To introduce an RM frame in $\mathbb{I}^3$, we need to look for an RM vector field in substitution to the principal normal. If $\mathbf{v}\perp\mathbf{t}$, we may write
\begin{equation}
\mathbf{v} = \mu\mathbf{n}+\nu\mathbf{b}\,,
\end{equation}
where $\mu\not=0$ (otherwise, $\mathbf{v}$ is just a multiple of $\mathbf{b}$). Now, imposing $\langle\mathbf{v},\mathbf{v}\rangle_z=1$ implies that
$
1=\mu^2\langle\mathbf{n},\mathbf{n}\rangle_z$ and then $\mu = \pm 1$.
The derivative of $\mathbf{v}$ is 
$
\mathbf{v}' 
= -\mu\kappa\,\mathbf{t}+(\mu\tau+\nu')\,\mathbf{b}$ and, assuming $\mathbf{v}$ to be an RM vector field, we have
\begin{equation}
\mathbf{v}'\parallel \mathbf{t} \Rightarrow \nu = -\mu\tau\Rightarrow \nu = -\mu \int \tau+\tau_0\,(\mbox{here } \mu=\pm1),
\end{equation}
with $\tau_0$ constant. Finally, imposing $\{\mathbf{t},\mathbf{v},\mathbf{b}\}$ has the same orientation as  $\{\mathbf{t},\mathbf{n},\mathbf{b}\}$, 
\begin{equation}
1 = \det(\mathbf{t},\mathbf{v},\mathbf{b}) =  \det(\mathbf{t},\mu\mathbf{n},\mathbf{b}) =\mu.
\end{equation}
\begin{remark}
Using Galilean trigonometric functions, i.e., $\cosg\,\phi=1$ and $\sing\,\phi=\phi$ \cite{Yaglom1979}, we can write an RM vector field $\mathbf{v}$ in terms of the Frenet frame as
\begin{equation}
\left\{
\begin{array}{c}
\mathbf{v}  =  \cosg(\theta)\,\mathbf{n}-\sing(\theta)\,\mathbf{b}\\
\theta'  =  \tau
\end{array}
\right.,
\end{equation}
in analogy to RM frames in Euclidean and Lorentz-Minkowski spaces \cite{BishopMonthly,OzdemirMJMS2008}.\label{remark::GalTrigFunc}
\end{remark}

From the discussion above it follows straightforwardly the 
\begin{theorem}
Let $\mathbf{n}_1$ be a unit normal vector field along $\alpha:I\to\mathbb{I}^3$. If $\mathbf{n}_1$ is RM and $\{\mathbf{t},\mathbf{n}_1,\mathbf{b}\}$ has the same orientation as the Frenet frame, then
\begin{equation}
 \mathbf{n}_1(s)=\mathbf{n}(s)-\left(\int_{s_0}^s \tau(x)\mathrm{d}x+\tau_0\right)\,\mathbf{b}(s)\,,
\end{equation}
where $\tau_0$ is a constant. In addition, a rotation minimizing frame $\{\mathbf{t},\mathbf{n}_1,\mathbf{n}_2=\mathbf{b}\}$ in  $\mathbb{I}^3$ satisfies
\begin{equation}
\frac{\mathrm{d}}{\mathrm{d}s}\left(
\begin{array}{c}
\mathbf{t}\\
\mathbf{n}_1\\
\mathbf{n}_2\\
\end{array}
\right)=\left(
\begin{array}{ccc}
0 & \kappa_1 & \kappa_2\\
-\kappa_1 & 0 & 0\\
0 & 0 & 0\\
\end{array}
\right)\left(
\begin{array}{c}
\mathbf{t}\\
\mathbf{n}_1\\
\mathbf{n}_2\\
\end{array}
\right),
\end{equation}
where the natural curvatures are $\kappa_1=\kappa=\kappa\,\cosg\,\theta$ and $\kappa_2=\kappa\,\theta=\kappa\,\sing\,\theta$, with  $\sing\,[\theta(s)]=\int_{s_0}^s \tau(x)\mathrm{d}x+\tau_0$.
\end{theorem}

We can relate the RM  curvatures $\kappa_1,\kappa_2$ to the Frenet ones $\kappa,\tau$ as
\begin{equation}
\left\{
\begin{array}{c}
\kappa_1(s) = \kappa(s)\,\cosg\,\theta(s)\\
\kappa_2(s) = \kappa(s)\,\sing\,\theta(s)\\
\end{array}
\right.\,\mbox{ and  }\,\theta'(s)= \tau(s),\label{eq::IsoRelBetweenRMandFrenet}
\end{equation}
which also shows that two RM frames differ by an additive constant, $\theta\mapsto\theta+\theta_0$, due to the action of the group $SOI(2)$ of plane isotropic rotations on the normal planes:
\begin{equation}
SOI(2)=\left\{M\in M_{2\times2}(\mathbb{R}):M=\left(
\begin{array}{cc}
\cosg\theta & 0\\
\sing\theta & \cosg\theta\\
\end{array}
\right)\right\}.
\end{equation}

This issue can be further clarified with the help of the ring of dual numbers $\mathbb{D}$ in the isotropic plane $\mathbb{I}^2$ \cite{Sachs1987}, since the normal plane is always isotropic (see the Appendix for the definition of $\mathbb{D}$). As in $\mathbb{E}^2$, where we may use a unit complex to describe a rotation\footnote{The same applies in $\mathbb{E}_1^2$ through the use of Lorentz numbers \cite{BirmanMonthly1984}: see subsection \ref{subsec::DescSemiIsoGeom} below.}, here we use a unit dual $p=1+\phi\,\varepsilon=\cosg\phi+\varepsilon\,\sing\phi$ to describe (Galilean) rotations in $\mathbb{I}^2$: $p\mapsto a\,p$ (see Fig. \ref{fig::DiagramSphPlaneCurv} in the Appendix). Indeed, identifying  $(x_1,y_1)\in\mathbb{I}^2$ with $x_1+y_1\varepsilon\in\mathbb{D}$,  a rigid motion in $\mathbb{I}^2$ is given by
\begin{equation} 
\left[
\begin{array}{c}
x_2\\
y_2\\
\end{array}
\right]=
\left[\begin{array}{cc}
1 & 0\\
\phi & 1\\
\end{array}
\right]\left[
\begin{array}{c}
x_1\\
y_1\\
\end{array}
\right]+
\left[
\begin{array}{c}
a\\
b\\
\end{array}
\right]=
\left[
\begin{array}{cc}
\cosg\phi & 0\\
\sing\phi & \cosg\phi\\
\end{array}
\right]\left[
\begin{array}{c}
x_1\\
y_1\\
\end{array}
\right]+
\left[
\begin{array}{c}
a\\
b\\
\end{array}
\right],
\end{equation}
where we used the linear (matrix) representation for $\mathbb{D}$ in Eq. (\ref{eq::LinearRepDualNumbers}).

In short, with the help of the ring of dual numbers $\mathbb{D}$, we can interpret an isotropic RM frame as a frame that minimizes isotropic (or Galilean) rotations. 

\subsection{Moving bivectors in simply isotropic space}

In $\mathbb{I}^3$ it is not possible to define a vector product with the same invariance significance as in Euclidean space. However, one can still do some interesting investigations by employing in $\mathbb{I}^3$ the usual vector product $\vecpe$ from  $\mathbb{E}^3$. Associated with the isotropic Frenet frame, one introduces a (moving) bivector frame as \cite{Sachs1990} 
\begin{equation}
\mathcal{T} = \mathbf{n}\vecpe\mathbf{b}=\tilde{\mathbf{t}},\,
\mathcal{N} = \mathbf{b}\vecpe\mathbf{t},\,\mbox{ and }\,
\mathcal{B} = \mathbf{t}\vecpe\mathbf{n},
\end{equation}
which results in a linearly independent frame, $\det(\mathcal{T},\mathcal{N},\mathcal{B})=\det(\mathbf{t},\mathbf{n},\mathbf{b})=1$ (\cite{Sachs1990}, Eqs. (7.43a-c), p. 130), and also leads to the equations
\begin{equation}
\mathcal{T}\,'=\kappa\,\mathcal{N},\,\mathcal{N}\,'=-\kappa\,\mathcal{T},\,\mbox{ and }\,\mathcal{B}\,'=-\tau\,\mathcal{N}.
\end{equation}

Analogously, we shall introduce  a (moving) RM bivector frame as
\begin{equation}
\left\{
\begin{array}{lcl}
\mathcal{T} &=& \mathbf{n}_1\vecpe\mathbf{n}_2 = \mathbf{n}\vecpe\mathbf{b}=\tilde{\mathbf{t}}\\
\mathcal{N}_1 &=& \mathbf{n}_2\vecpe\mathbf{t}=\mathcal{N}\\
\mathcal{N}_2 &=& \mathbf{t}\vecpe\mathbf{n}_1
\end{array}
\right.\,,
\end{equation}
which satisfies
\begin{proposition}
The moving frame $\{\mathcal{T},\mathcal{N}_1,\mathcal{N}_2\}$ forms a basis for $\mathbb{R}^3$. In addition, a moving RM bivector frame satisfies the equation
\begin{equation}
\frac{\mathrm{d}}{\mathrm{d}s}\left(
\begin{array}{c}
\mathcal{T}\\
\mathcal{N}_1\\
\mathcal{N}_2\\
\end{array}
\right)=\left(
\begin{array}{ccr}
0 & \kappa_1 & \,\,0\\
-\kappa_1 & 0 & \,\,0\\
-\kappa_2 & 0 & \,\,0\\
\end{array}
\right)\left(
\begin{array}{c}
\mathcal{T}\\
\mathcal{N}_1\\
\mathcal{N}_2\\
\end{array}
\right)\,,
\end{equation}
where $\kappa_1=\kappa=\kappa\,\cosg\,\theta$, $\kappa_2=\kappa\,\theta=\kappa\,\sing\,\theta$, and $\sing\,(\theta)=\int \tau+\tau_0$.
\end{proposition}

\section{Simply isotropic spherical curves}

\subsection{Isotropic osculating spheres}

Due to the degeneracy of the isotropic metric, some geometric concepts can not be defined in $\mathbb{I}^3$ by just using 
$\langle\cdot,\cdot\rangle_z$. This is the case for spheres. We define \emph{simply isotropic spheres} as connected and irreducible surfaces of degree 2 given by the 4-parameter family
$
(x^2+y^2)+2c_1x+2c_2y+2c_3z+c_4=0,$ 
 where $c_i\in\mathbb{R}$ \cite{Sachs1990}, p. 66. In addition, up to a rigid motion (in $\mathbb{I}^3$), we can express a sphere in one of the two normal forms below
\begin{enumerate}
\item sphere of parabolic type: $
z = \frac{1}{2p}(x^2+y^2)\,\mbox{ with }\,p\not=0
$; and 
\item sphere of cylindrical type: $x^2+y^2=r^2\,\mbox{ with }\,r>0.$
\end{enumerate}

The quantities $p$ and $r$ are isotropic invariants. Moreover, spheres of cylindrical type are precisely the set of points equidistant from a given center: $\langle\mathbf{x}-P,\mathbf{x}-P\rangle_z = r^2.$ Notice however, that the center $P$ of a cylindrical sphere is not uniquely defined, any other $Q$ with the same top view as $P$, i.e., $\tilde{Q}=\tilde{P}$, would do the job. We can remedy this by assuming $P=(x,y,0)$.

An osculating sphere of an admissible curve $\alpha$ at a point $\alpha(s_0)$ is the (isotropic) sphere having contact of order 3 with $\alpha$. Its position vector $\mathbf{x}$ satisfies
\begin{equation}
\lambda\langle\mathbf{x}-\alpha_0,\mathbf{x}-\alpha_0\rangle_z+\langle\mathbf{u},\mathbf{x}-\alpha_0\rangle = 0\,,\label{eq::IsoOscSphere}
\end{equation}
where $\alpha_0=\alpha(s_0)$, $\langle\cdot,\cdot\rangle$ is the inner product in  $\mathbb{E}^3$, and $\lambda\in\mathbb{R}$ and $\mathbf{u}\in\mathbb{R}^3$ are constants to be determined, Eq. (7.18) of  \cite{Sachs1990}.

\subsection{Characterization of spherical curves in simply isotropic space}

Our approach to spherical curves is based on order of contact. More precisely, we investigate osculating spheres in $\mathbb{I}^3$ by using RM frames and their associated bivector frames. Then, we use that a curve is spherical when its osculating spheres are all equal to the sphere that contains the curve (see proof of Theorem \ref{theo:BishopCharacSpherericalCurves}). 

Defining $F(\mathbf{x})=\lambda\langle\mathbf{x}-\alpha_0,\mathbf{x}-\alpha_0\rangle_z+\langle\mathbf{u},\mathbf{x}-\alpha_0\rangle$, where $\alpha_0=\alpha(s_0)$ and $\lambda,\mathbf{u}$ are constants to be determined, we have for the derivatives of $F\,\circ\,\alpha$
\begin{eqnarray}
\left\{
\begin{array}{ccc}
F' & = & 2\lambda\langle\alpha-\alpha_0,\mathbf{t}\rangle_z+\langle\mathbf{u},\mathbf{t}\rangle,\\[4pt]
F'' & = & 2\lambda\langle\mathbf{t},\mathbf{t}\rangle_z+2\lambda\langle\alpha-\alpha_0,\sum_i\kappa_i\mathbf{n}_i\rangle_z+\langle\mathbf{u},\sum_i\kappa_i\mathbf{n}_i\rangle\\[4pt]
F''' & = & 2\lambda\langle\alpha-\alpha_0,-\kappa_1^2\mathbf{t}+\sum_i\kappa_i'\mathbf{n}_i\rangle_z+\langle\mathbf{u},-\kappa_1^2\mathbf{t}+\sum_i\kappa_i'\mathbf{n}_i\rangle\\
\end{array}
\right..
\end{eqnarray}
Imposing contact of order 3, $(F\circ\alpha)'=(F\circ\alpha)''=(F\circ\alpha)'''=0$ at $s_0$, gives
\begin{equation}
\left\{
\begin{array}{c}
\langle\mathbf{u},\mathbf{t}(s_0)\rangle = 0\\
2\lambda = - \langle\mathbf{u},\sum_i\kappa_i(s_0)\mathbf{n}_i(s_0)\rangle\\
\langle\mathbf{u},\sum_i\kappa_i'(s_0)\mathbf{n}_i(s_0)\rangle = 0\\
\end{array}
\right..
\end{equation}
From the first and third equations above, we find that
\begin{equation}
\mathbf{u}=\rho\,[\mathbf{t}\vecpe(\kappa_1'\mathbf{n}_1+\kappa_2'\mathbf{n}_2)](s_0)=\rho\,[\kappa_1'\mathcal{N}_2-\kappa_2'\mathcal{N}_1](s_0),
\end{equation}
for some constant $\rho\not=0$. On the other hand, from the second equation we find
\begin{equation}
2\lambda+\rho\,[\kappa_1'\kappa_2\langle\mathbf{n}_2,\mathcal{N}_2\rangle-\kappa_1\kappa_2'\langle\mathbf{n}_1,\mathcal{N}_1\rangle](s_0)=0.
\end{equation}
The reader can easily verify that $\langle\mathbf{n}_i,\mathcal{N}_i\rangle = \det(\mathbf{t},\mathbf{n}_1,\mathbf{n}_2)=1$, and then we can rewrite the expression above as
\begin{equation}
2\lambda = \rho\,[\kappa_1\kappa_2'-\kappa_1'\kappa_2](s_0)=\rho\,\tau(s_0)\kappa^2(s_0),
\end{equation}
where we have used the expressions of $(\kappa_1,\kappa_2)$ in terms of $(\kappa,\tau)$, Eq.  (\ref{eq::IsoRelBetweenRMandFrenet}).

In short, the equation for the isotropic osculating sphere (\ref{eq::IsoOscSphere}), with respect to an RM frame and its associated bivector frame, can be written as
\begin{equation}
\tilde{\mathbf{x}}^2-2\left\langle\mathbf{x},\tilde{\alpha}_0+\frac{\kappa_2'\mathcal{N}_1-\kappa_1'\mathcal{N}_2}{\tau\kappa^2}\vert_{s_0}\right\rangle+2\left[\frac{\tilde{\alpha}_0^2}{2}-\left\langle\alpha_0,\frac{\kappa_1'\mathcal{N}_2-\kappa_2'\mathcal{N}_1}{\tau\kappa^2}\vert_{s_0}\right\rangle\right]=0\,,
\end{equation}
where $\tilde{\mathbf{x}}^2=\langle\mathbf{x},\mathbf{x}\rangle_z$.
\begin{theorem}
An admissible $C^4$ regular curve $\alpha:I\to \mathbb{I}^3$ lies on the surface of a sphere if and only if its normal development, i.e., the curve $(\kappa_1(s),\kappa_2(s))$, lies on a line not passing through the origin. In addition, $\alpha$ is a spherical curve of cylindrical type with radius $r$ if and only if $\kappa$ is constant and equal to $r^{-1}$. 
\end{theorem}
\begin{proof}
The condition of being spherical implies that the isotropic osculating spheres are all the same and equal to the sphere that contains the curve. Then
\begin{equation}
\frac{\mathrm{d}}{\mathrm{d}s}\left[\tilde{\alpha}+\frac{\kappa_2'\mathcal{N}_1}{\tau\kappa^2}-\frac{\kappa_1'\mathcal{N}_2}{\tau\kappa^2}\right]=0\label{eq::CondToBeSphericalVector}
\end{equation}
and
\begin{equation}
\frac{\mathrm{d}}{\mathrm{d}s}\left[\frac{\tilde{\alpha}^2}{2}-\left\langle\alpha,\frac{\kappa_1'\mathcal{N}_2-\kappa_2'\mathcal{N}_1}{\tau\kappa^2}\right\rangle\right]=\frac{\mathrm{d}}{\mathrm{d}s}\,\left\langle\alpha,\frac{\tilde{\alpha}}{2}-\frac{\kappa_1'\mathcal{N}_2-\kappa_2'\mathcal{N}_1}{\tau\kappa^2}\right\rangle=0\,.\label{eq::CondToBeSphericalScalar}
\end{equation}

The first condition gives
\begin{eqnarray}
0 & = & \tilde{\mathbf{t}}+\left(\frac{\kappa_2'}{\tau\kappa^2}\right)'\mathcal{N}_1-\left(\frac{\kappa_1'}{\tau\kappa^2}\right)'\mathcal{N}_2+\frac{\kappa_2'}{\tau\kappa^2}(-\kappa_1\mathcal{T})-\frac{\kappa_1'}{\tau\kappa^2}(-\kappa_2\mathcal{T})\nonumber\\
& = & \left(\frac{\kappa_2'}{\tau\kappa^2}\right)'\mathcal{N}_1-\left(\frac{\kappa_1'}{\tau\kappa^2}\right)'\mathcal{N}_2,
\end{eqnarray}
which, by taking into account the linear independence of $\{\mathcal{N}_1,\mathcal{N}_2\}$, implies
\begin{equation}
a_1:=-\frac{\kappa_2'}{\tau\kappa^2}=\mbox{constant};\,a_2:=\frac{\kappa_1'}{\tau\kappa^2}=\mbox{constant}.
\end{equation}
On the other hand, condition (\ref{eq::CondToBeSphericalScalar}) implies
\begin{equation}
0 = \langle\alpha,\frac{\mathrm{d}}{\mathrm{d}s}\left(\frac{\tilde{\alpha}}{2}-\sum_ia_i\mathcal{N}_i\right)\rangle+\langle\mathbf{t},\frac{\tilde{\alpha}}{2}-\sum_ia_i\mathcal{N}_i\rangle= (1+a_1\kappa_1+a_2\kappa_2)\langle\alpha,\mathbf{t}\rangle_z,
\end{equation}
where we used that $\langle\tilde{\alpha},\mathbf{t}\rangle=\langle\alpha,\tilde{\mathbf{t}}\rangle=\langle\alpha,\mathbf{t}\rangle_z$ to obtain the second equality. If $\alpha$ is not of cylindrical type, $\langle\alpha,\alpha\rangle_z$ is not a constant, i.e., $\langle\alpha,\mathbf{t}\rangle_z\not=0$. Then, for a parabolic spherical curve, $(\kappa_1,\kappa_2)$ lies on a line not passing through the origin.

On the other hand, if $\alpha$ is of cylindrical type $\langle\alpha(s)-P,\alpha(s)-P\rangle_z=r^2$, then
\begin{equation}
\langle\mathbf{t},\alpha-P\rangle_z=0\Rightarrow\alpha-P=a_1\mathbf{n}_1+a_2\mathbf{n}_2\,.\label{eq::InnerProdTangwithalphaP}
\end{equation}
Here, $a_1=\langle\alpha-P,\mathbf{n}_1\rangle_z\Rightarrow a_1'=\langle\mathbf{t},\mathbf{n}_1\rangle_z+\langle\alpha-P,-\kappa_1\mathbf{t}\rangle_z=0$ and $a_1$ is a constant.

Taking the derivative of Eq. (\ref{eq::InnerProdTangwithalphaP}) gives
\begin{eqnarray}
0 & = & \langle\mathbf{t},\mathbf{t}\rangle_z+\langle\kappa_1\mathbf{n}_1+\kappa_2\mathbf{n}_2,\alpha-P\rangle_z=  1+a_1\kappa.
\end{eqnarray}
Hence, the curvature $\kappa=\kappa_1$ is a constant and, in addition, $r^2=\langle\alpha-P,\alpha-P\rangle_z=\langle a_1\mathbf{n}_1+a_2\mathbf{n}_2,a_1\mathbf{n}_1+a_2\mathbf{n}_2\rangle_z=a_1^2$.

Reciprocally, if $\kappa$ is a (non-zero) constant, define $P=\alpha+\kappa^{-1}\mathbf{n}_1$. Taking the derivative gives $P'=\mathbf{t}+\kappa^{-1}(-\kappa\mathbf{t})=0$ and then $P$ is a constant. Clearly we have $\langle\alpha-P,\alpha-P\rangle_z=1/\kappa^2$. 
\end{proof}
\begin{remark}
In the proof above, we could also use the Frenet frame instead of an RM one. In this case, spherical curves are characterized by $\kappa'/(\kappa^2\tau)=\mbox{constant}$.
\end{remark}
\begin{proposition}
An admissible curve $\alpha:I\to \mathbb{I}^3$ lies on a plane if and only if its normal development $(\kappa_1(s),\kappa_2(s))$ lies on a line passing through the origin. 
\end{proposition}
\begin{proof}
A curve lies on a plane $\Pi$ if and only if all its osculating planes are equal to $\Pi$. Define $F(\mathbf{x})=\langle\mathbf{x}-\alpha_0,\mathbf{u}\rangle$, where $\langle\mathbf{u},\mathbf{u}\rangle=1$ (for convenience, we describe a plane in $\mathbb{I}^3$ through a unit vector with respect to $\mathbb{E}^3$). Taking the derivatives of $F\circ\alpha$ twice and demanding a contact of order 2, we have
\begin{equation}
\left\{
\begin{array}{ccc}
 \langle\mathbf{t}(s_0),\mathbf{u}\rangle & = & 0\\
 \langle[\kappa_1\mathbf{n}_1+\kappa_2\mathbf{n}_2]\vert_{s_0},\mathbf{u}\rangle & = & 0\\
\end{array}
\right..
\end{equation}
From these equations we deduce that
\begin{equation}
\mathbf{u}=\mathbf{u}(s_0) = \rho(s_0)\,[\mathbf{t}\times_e(\kappa_1\mathbf{n}_1+\kappa_2\mathbf{n}_2)]\vert_{s_0}=\rho(s_0)\,[\kappa_1\mathcal{N}_2-\kappa_2\mathcal{N}_1]\vert_{s_0}\,,
\end{equation}
where, by applying the definition of the Frenet and RM bivectors, we can write $\rho=(\kappa_1\Vert\mathcal{B}\Vert)^{-1}$. 

The condition of being a plane curve is equivalent to $\mathrm{d}\mathbf{u}/\mathrm{d}s=0$. Thus
\begin{eqnarray}
\frac{\mathrm{d}\mathbf{u}}{\mathrm{d}s} & = & -\frac{1}{\kappa_1\Vert\mathcal{B}\Vert}\left[\left(\frac{\kappa_1\kappa_2'-\kappa_1'\kappa_2}{\kappa_1}+\frac{\kappa_2\tau\langle\mathcal{B},\mathcal{H}\rangle}{\langle\mathcal{B},\mathcal{B}\rangle}\right)\mathcal{N}_1+\frac{\kappa_1\tau\langle\mathcal{B},\mathcal{H}\rangle}{\langle\mathcal{B},\mathcal{B}\rangle}\mathcal{N}_2\right]\nonumber\\
& = &-\frac{\tau}{\Vert\mathcal{B}\Vert}\left[\left(1+\frac{\theta\langle\mathcal{B},\mathcal{H}\rangle}{\langle\mathcal{B},\mathcal{B}\rangle}\right)\mathcal{N}_1+\frac{\langle\mathcal{B},\mathcal{H}\rangle}{\langle\mathcal{B},\mathcal{B}\rangle}\mathcal{N}_2\right],
\end{eqnarray}
where we used that $\tau\kappa^2=\kappa_1\kappa_2'-\kappa_1'\kappa_2$, $\mathcal{N}_1=\mathcal{H}$, $\langle\mathcal{H},\mathcal{H}\rangle=1$, and $\langle\mathcal{T},\mathcal{T}\rangle=1$.

Finally, it is easy to see that the planarity condition, i.e., $\mathbf{u}'=0$, is equivalent to $\tau=0\Leftrightarrow (\kappa_2/\kappa_1)'=\kappa_1\kappa_2'-\kappa_1'\kappa_2=0$, which  is equivalent to $\kappa_2/\kappa_1=\mbox{const.}$ and then implies $(\kappa_1,\kappa_2)$ lies on a line through the origin.
\end{proof}

\section{Differential geometry in pseudo-isotropic space}

Following the Cayley-Klein paradigm, we must specify an absolute figure in order to build the pseudo-isotropic space. Here, the pseudo-isotropic absolute is composed by a plane at infinity, identified with $x_0=0$, and a degenerate quadric of index one, identified with $x_0^2+x_1^2-x_2^2=0$ (there are other choices for the pseudo-isotropic absolute figure and we discuss it in the next subsection). Equivalently, we may say that in homogeneous coordinates the pseudo-isotropic absolute figure is composed by a plane $\omega:x_0=0$ and a pair of lines $f_1:0=x_0=x_1+x_2$ and $f_2:0=x_0=x_1-x_2$. Observe, in addition, that the point $F=[0:0:0:1]\in\mathbb{P}^3$ lies in the intersection $f_1\cap f_2$ and, therefore, should be preserved. Hence, the pseudo-isotropic absolute figure is alternatively given by $\{\omega,f_1,f_2,F\}$.

Let us denote a projectivity in $\mathbb{P}^3$ by
\begin{equation}
\left\{
\begin{array}{ccccccccc}
\bar{x}_0 & = & a_{00}x_0 &+& a_{01}x_1 &+& a_{02}x_2 &+& a_{03}x_{3}\\
\bar{x}_1 & = & a_{10}x_0 &+& a_{11}x_1 &+& a_{12}x_2 &+& a_{13}x_{3}\\
\bar{x}_2 & = & a_{20}x_0 &+& a_{21}x_1 &+& a_{22}x_2 &+& a_{23}x_{3}\\
\bar{x}_3 & = & a_{30}x_0 &+& a_{31}x_1 &+& a_{32}x_2 &+& a_{33}x_{3}\\
\end{array}
\right.,\, \det(a_{ij})\not=0.
\end{equation}

Imposing that $\omega$ and $F$ should be preserved leads to $a_{01}=a_{02}=a_{03}=0$ and $a_{13}=a_{23}=0$, respectively. A projectivity that preserves the absolute figure is said to be a \emph{direct projectivity} if it takes $f_i$ to $f_i$, i.e., $x_1\pm x_2=0$ goes in $\bar{x}_1\pm \bar{x}_2=0$, and an \emph{indirect projectivity} if it takes $f_i$ to $f_j$ ($i\not=j$), i.e., $x_1\pm x_2=0$ goes in $\bar{x}_1\mp \bar{x}_2=0$. The coefficients $a_{ij}$ of a direct projectivity should satisfy the relations
\begin{equation}
\left\{
\begin{array}{c}
a_{11}-a_{12}+a_{21}-a_{22} = 0\\
a_{11}+a_{12}-a_{21}-a_{22} = 0\\
\end{array}
\right..
\end{equation}
Adding and subtracting the equations above leads to $a_{11}=a_{22}$ and $a_{12}=a_{21}$.

Finally, going to affine coordinates and denoting $a:=a_{10}/a_{00}$, $b:=a_{20}/a_{00}$, $c:=a_{30}/a_{00}$, $p\cosh\phi:=a_{11}/a_{00}$, $p\sinh\phi:=a_{12}/a_{00}$, and $c_i:=a_{3i}/a_{00}$ ($i=1,2,3$), defines the group $G_8^{\mathrm{p}}$ of pseudo-isotropic direct similarities
\begin{equation}
\left\{
\begin{array}{ccc}
\bar{x} & = & a + p(x \cosh\phi + y \sinh\phi) \\
\bar{y} & = & b + p(x \sinh\phi + y \cosh\phi) \\
z & = & c + c_1 x + c_2 y + c_{3} z\\
\end{array}
\right..\label{eq::SemiIsoDirSim}
\end{equation}

Let us introduce a metric in $\mathbb{I}_{\mathrm{p}}^3=\mathbb{P}^3/\omega$ according to
\begin{equation}
\langle \mathbf{u},\mathbf{v}\rangle_{z,{\mathrm{p}}} = u_1v_1 - u_2v_2. \label{eq::DefSemiIsoMetr}
\end{equation}
If we apply a transformation from $G_8^{\mathrm{p}}$ to $A,B\in\mathbb{I}_{\mathrm{p}}^3$, the norm $\Vert\mathbf{v}\Vert_{\mathrm{p}}=\sqrt{\vert\langle\mathbf{v},\mathbf{v}\rangle_{z,{\mathrm{p}}}\vert}$ induced by the metric above satisfies
\begin{equation}
\Vert\bar{B}-\bar{A}\,\Vert_{\mathrm{p}} = p\, \Vert B-A\,\Vert_{\mathrm{p}}.
\end{equation}

For $p=1$, the pseudo-isotropic metric $\langle\cdot,\cdot\rangle_{z,{\mathrm{p}}}$ is an absolute invariant. Note in addition that, as happens in the simply isotropic space, the distance between two points with the same top view\footnote{As in $\mathbb{I}^3$, we may define the \emph{top view} as the projection on the $xy$-plane, \emph{(pseudo-)isotropic direction} as $(0,0,z)$, which is preserved by $\mathcal{B}_6^{\mathrm{p}}$, \emph{(pseudo-)isotropic lines} as those lines with isotropic direction, and \emph{(pseudo-)isotropic planes} as those planes containing an isotropic line.} vanishes. In such cases, one may introduce a \emph{pseudo-isotropic codistance} as
\begin{equation}
\mbox{cd}_{z,{\mathrm{p}}}(A,B) = \vert b_3-a_3\vert.
\end{equation}
Applying a transformation from $G_8^{\mathrm{p}}$ to $A,B\in\mathbb{I}_{\mathrm{p}}^3$ leads to 
\begin{equation}
\mbox{cd}_{z,{\mathrm{p}}}(\bar{A},\bar{B})=p\,\mbox{cd}_{z,{\mathrm{p}}}(A,B).
\end{equation}
\begin{definition}
The group of $\mbox{d}_{z,{\mathrm{p}}}$-length and $\mbox{cd}_{z,{\mathrm{p}}}$-colength preserving direct projectivities forms the group of pseudo-isotropic (rigid) motions $\mathcal{B}_6^{\mathrm{p}}$. The pseudo-isotropic geometry is the study of $(\mathbb{I}_{\mathrm{p}}^3,\mathcal{B}_6^{\mathrm{p}})$.
\end{definition}
\begin{remark}
In $\mathbb{I}_{\mathrm{p}}^3$ one has $A=B\Leftrightarrow\mbox{d}_{z,{\mathrm{p}}}(A,B)=0$ and $\mbox{cd}_{z,{\mathrm{p}}}(A,B)=0$.
\end{remark}

In short, pseudo-isotropic geometry is the study of those properties in $\mathbb{R}^3$ invariant by the action of the 6-parameter group $\mathcal{B}_6^{\mathrm{p}}$
\begin{equation}
\left\{
\begin{array}{ccc}
\bar{x} & = & a + x\,\cosh\phi+y\,\sinh\phi\\ 
\bar{y} & = & b + x\,\sinh\phi+y\,\cosh\phi\\
\bar{z} & = & c + c_1x+c_2y+z\\
\end{array}
\right.,\,a,b,c,c_1,c_2,\phi\in\mathbb{R}.\label{eq::SemiIsoGroupSB6}
\end{equation}

Notice that on the top view plane, the pseudo-isotropic geometry behaves like the geometry in  $\mathbb{E}^2_1$. Indeed, up to translations, the action of $\mathcal{B}_6^{\mathrm{p}}$ on the top view corresponds to the action of $O_1^{++}(2)$ \cite{LopesIEJG2014}, the group of (hyperbolic) rotations in $\mathbb{E}_1^2$ that preserves both the orientation of $\mathbb{R}^2$ as a vector space, i.e., $O_1^{++}(2)\subset SO_1(2)$, and the time-orientation of $\mathbb{E}_1^2$, i.e., $O_1^{++}(2)\subset O_1^+(2)$.  

\begin{remark}
The group of isometries of $\mathbb{E}_1^2$ has  four components: $O_1^{++}(2)$, $O_1^{+-}(2)$, $O_1^{-+}(2)$, and $O_1^{--}(2)$ (a $+$ sign in the 1st upper position means that the vector space orientation is preserved, while in the 2nd upper position it means that time-orientation is preserved; a minus sign means that orientation is not preserved) \cite{LopesIEJG2014}. Choosing ``$-\cosh\phi$'' for the $x$-coefficient in Eq. (\ref{eq::SemiIsoDirSim}) leads to the action of $O_1^{+-}(2)$ on the top view. On the other hand, the study of indirect projectivities gives
\begin{equation}
\left\{
\begin{array}{ccc}
\bar{x} & = & a \pm x \cosh\phi - y \sinh\phi \\
\bar{y} & = & b + x \sinh\phi \mp y \cosh\phi \\
z & = & c + c_1 x + c_2 y + z\\
\end{array}
\right..
\end{equation}
These projectivities correspond to the action of $O_1^{-\pm}(2)$ on the top view. However, $O_1^{-+}(2)\cup O_1^{--}(2)$ does not form a group, since it does not contain the identity.
\end{remark}

\subsection{Alternative descriptions of pseudo-isotropic geometry}
\label{subsec::DescSemiIsoGeom}

In some pioneering works \cite{Sachs1990,StrubeckerCrelle1938}, the absolute figure of the pseudo-isotropic space is given in homogeneous coordinates by  $\omega:x_0=0$ together with the pair of real lines $x_0=x_1=0$ and $x_0=x_2=0$, which leads to the 6 parameter group
\begin{equation}
\left\{
\begin{array}{ccccccccc}
\bar{x} & = & a & + & p\,x & &           & &  \\ 
\bar{y} & = & b & + &      & & p^{-1}\,y & &  \\
\bar{z} & = & c & + & c_1x &+& c_2\,y      &+& z\\
\end{array}
\right.,\,a,b,c,c_1,c_2,p\in\mathbb{R}:\label{eq::StrubeckerSemiisoRigMtn}
\end{equation}
see \cite{StrubeckerCrelle1938}, Eqs. (1), (2), (4), and (10), pp. 136-137; or \cite{Sachs1990}, Eqs. (1.68), (1.70), p. 24.

This choice furnishes a geometry equivalent to that described by $\mathcal{B}_6^{\mathrm{p}}$, Eq. (\ref{eq::SemiIsoGroupSB6}). (Notice that here the isotropic metric changes to $\mathrm{d}s^2=\mathrm{d}x\,\mathrm{d}y$.) Indeed, this can be made clear with the help of Lorentz numbers $\mathbb{L}$, also known as double or hyperbolic numbers (see Appendix). Rotations in $\mathbb{E}^2$ may be described through multiplication by a unit spacelike Lorentz number $a$, i.e., $q\mapsto a\,q$ \cite{BirmanMonthly1984}. A unit Lorentz number is written as $a=p\,\mathrm{e}_++p^{-1}\,\mathrm{e}_-$ in the basis $\{\mathrm{e}_{\pm}\}$ and as $a=\cosh\phi+\ell\sinh\phi$ in the basis $\{1,\ell\}$ (see Fig. 1 in the Appendix). It follows that a rotation may be equivalently written as  
\begin{equation}
\left[
\begin{array}{c}
x_2\\
y_2\\
\end{array}
\right]=
\left[
\begin{array}{cc}
\cosh\phi & \sinh\phi\\
\sinh\phi & \cosh\phi\\
\end{array}
\right]\left[
\begin{array}{c}
x_1\\
y_1\\
\end{array}
\right]\mbox{ or }\left[
\begin{array}{c}
x_2\\
y_2\\
\end{array}
\right]=
\left[
\begin{array}{cc}
p & 0\\
0 & p^{-1}\\
\end{array}
\right]\left[
\begin{array}{c}
x_1\\
y_1\\
\end{array}
\right],
\end{equation}
where we used the linear representations in Eqs. (\ref{eq::LinearRepLorentzNumbersNonNullBasis}) and (\ref{eq::LinearRepLorentzNumbersNullBasis}), respectively.

In short, our expression for $\mathcal{B}_6^{\mathrm{p}}$ in Eq. (\ref{eq::SemiIsoGroupSB6}) and that of Strubecker \cite{StrubeckerCrelle1938} in Eq. (\ref{eq::StrubeckerSemiisoRigMtn}) are equivalent, the choice between them being just a matter of convenience. 

\section{pseudo-isotropic spheres}

A pseudo-isotropic sphere is a connected and irreducible surface of degree 2 that contains the absolute figure. As we will see below, the pseudo-isotropic spheres are given by the 4-parameter family
\begin{equation}
(x^2-y^2)+2c_1x+2c_2y+2c_3z+c_4=0,\,c_1,c_2,c_3,c_4\in\mathbb{R}.\label{eq::SemiIsoSpheres}
\end{equation}
In addition, up to a rigid motion (in $\mathbb{I}_{\mathrm{p}}^3$), we can express a sphere in one of the two normal forms below:
\begin{equation}
\mbox{(sphere of parabolic type) }z = \frac{1}{2p}(x^2-y^2)\,\mbox{ with }\,p\not=0;\label{eq::SemiIsoParaSph}
\end{equation}
and
\begin{equation}
\mbox{(sphere of cylindrical type) }x^2-y^2=\pm\, r^2\,\mbox{ with }\,r>0.\label{eq::SemiIsoCylSph}
\end{equation}
\begin{remark}
In $\mathbb{E}^3$ these equations define a hyperbolic paraboloid and a hyperbolic cylinder, respectively, which justifies the names for the normal forms.
\end{remark}

A degree 2 surface in $\mathbb{P}^3$ may be written as
$
\mathcal{Q}:\sum_{i,j=0}^3c_{ij}x_ix_j=0\,,
$ 
where $c_{ij}\in\mathbb{R}$. If $\mathcal{Q}$ contains the absolute, i.e.,  $x_0=x_1\pm x_2=0\Rightarrow \mathcal{Q}(x_0,\dots,x_3)=0$, then
\begin{equation}
\left\{
\begin{array}{c}
(c_{11}+2c_{12}+c_{22})x_2^2+c_{33}x_3^2+2(c_{23}+c_{13})x_2x_3 =0 \\
(c_{11}-2c_{12}+c_{22})x_2^2+c_{33}x_3^2+2(c_{23}-c_{13})x_2x_3 =0 \\
\end{array}
\right..
\end{equation}
Since the above equation must be satisfied for all $[x_2:x_3]\in\mathbb{P}^1$, we conclude that
\begin{equation}
\left\{
\begin{array}{c}
c_{11}\pm2c_{12}+c_{22}=0\\
c_{23}\pm c_{13} = 0\\
c_{33} =0
\end{array}
\right.\Leftrightarrow \left\{
\begin{array}{c}
c_{12}=c_{13}=c_{23}=0\\
c_{11}+c_{22} = 0\\
\end{array}
\right..
\end{equation}
Finally, going to affine coordinates gives
\begin{equation}
c_{11}(x^2-y^2)+2c_{01}x+2c_{02}y+2c_{03}z+c_{00}=0,
\end{equation}
where we must have $c_{11}\not=0$. If $c_{03}\not=0$, we may write the equation above as
\begin{equation}
z=R(x^2-y^2)+ax+by+cz+d.
\end{equation}
The sphere above can be written in the parabolic normal form, Eq. (\ref{eq::SemiIsoParaSph}) after a convenient pseudo-isotropic rigid motion. On the other hand, if $c_{03}=0$, then, after a convenient semi-isotropic rigid motion, we have a cylindrical sphere, Eq. (\ref{eq::SemiIsoCylSph}). 

\section{Moving frames along curves in pseudo-isotropic space}

A curve $\alpha:I\to\mathbb{I}_{\mathrm{p}}^3$ is said to be \emph{regular} if $\alpha'\not=0$. As in $\mathbb{I}^3$, $\alpha'(t_0)$ is an \emph{inflection point} if $\{\alpha'(t_0),\alpha''(t_0)\}$ is linearly dependent, i.e., $\exists\,\sigma\in\mathbb{R}$ such that $\alpha''(t_0)=\sigma\alpha'(t_0)$. Notice that being regular is a (pseudo-isotropic) geometric concept, i.e.,  $\forall\,T\in \mathcal{B}_6^{\mathrm{p}}$, $\alpha'(t)\not=0\Rightarrow (T\circ\alpha)'(t)\not=0$. The same for an inflection point.

To find the osculating plane outside an inflection point $\alpha(t_0)$, we may employ the inner product in Lorentz-Minkowski space $\mathbb{E}_1^3$ given by $\langle \mathbf{u},\mathbf{v}\rangle_1=u_1v_1-u_2v_2+u_3v_3$. Defining $F(\mathbf{x})=\langle\mathbf{x},\mathbf{u}\rangle_1$, where  $\Vert\mathbf{u}\Vert_1=\sqrt{\vert\langle\mathbf{u},\mathbf{u}\rangle_1\vert}=1$, and imposing the order 2 contact ($F'=F''=0$ at $\alpha(t_0)$) leads to $
\mathbf{u} = \rho\,(\alpha'\times_1\alpha''),\,\rho\not=0$, where $\times_1$ is the vector product in $\mathbb{E}_1^3$: $\mathbf{v}\times_1\mathbf{w}=\det[(\mathbf{i},v_1,w_1),(-\mathbf{j},v_2,w_2),(\mathbf{k},v_3,w_3)]$. Thus, as expected, the position vector $\mathbf{x}$ of the osculating plane at $\alpha(t_0)$ verifies
\begin{equation}
\langle\,\mathbf{x}-\alpha(t_0),\alpha'(t_0)\times_1\alpha''(t_0)\,\rangle_1 = 0.\label{eq::SemiIsoOscPlane}
\end{equation}

\begin{definition}
A regular curve $\alpha:I\to\mathbb{I}_{\mathrm{p}}^3$ free of inflection points is an \emph{admissible curve} if all the osculating planes are not pseudo-isotropic. Equivalently, $(x'y''-x''y')\vert_{t}\not=0$ for all $t\in I$, where $\alpha(t)=(x(t),y(t),z(t))$ (notice, $\mbox{span}\{\alpha',\alpha''\}$ is isotropic if and only the third coordinate of $\alpha'\times_1\,\alpha''$ vanishes).
\end{definition}

The concept of reparametrization and arc-length parameter are defined as usual. Note however that curves in $\mathbb{I}_{\mathrm{p}}^3$ may have distinct causal characters: a vector $\mathbf{v}$ is said to be \emph{spacelike} if $\langle\mathbf{v},\mathbf{v}\rangle_{z,{\mathrm{p}}}>0$ or $\mathbf{v}=0$, \emph{timelike} if $\langle\mathbf{v},\mathbf{v}\rangle_{z,{\mathrm{p}}}<0$, and \emph{lightlike} if  $\langle\mathbf{v},\mathbf{v}\rangle_{z,{\mathrm{p}}}=0$ and $\mathbf{v}\not=0$. The causal character of $\alpha$ is given by that of $\alpha'$.

A lightlike curve $\alpha$ gives rise to a top view curve $\tilde{\alpha}$ in $\mathbb{E}_1^2$ whose image must lie on a straight line: $x=\pm y$ is the light cone in $\mathbb{E}_1^2$. These curves are not admissible: the light cone in $\mathbb{I}_{\mathrm{p}}^3$ is the set of pseudo-isotropic planes $(\mu,\pm\mu,\nu)$, $\mu,\nu\in\mathbb{R}$. So, in our study we shall restrict ourselves to space- and time-like curves (in principle, a curve may change its causal character. We shall not consider this here, but the interested reader may consult \cite{SaloomGD2012}: please, observe that their notation for the metric and curvature in $\mathbb{E}_1^2$ is slightly distinct from ours). Finally, since in $\mathbb{E}_1^2$ a vector $(x,y)$ is spacelike (timelike) if and only if $(y,x)\perp(x,y)$ is timelike (spacelike), we do not have non-lightlike curves with a lightlike acceleration vector. 


\subsection{Pseudo-isotropic Frenet frame}

Let $\alpha:I\to\mathbb{I}_{\mathrm{p}}^3$ be a unit speed admissible curve. Let us introduce $\epsilon=\langle\mathbf{t},\mathbf{t}\rangle_{z,{\mathrm{p}}}\in\{-1,+1\}$. If $\mathbf{t}'\not=0$, we define the pseudo-isotropic principal normal vector and curvature function, respectively, as
\begin{equation}
\mathbf{n}(s)=\frac{\mathbf{t}'(s)}{\eta\,\kappa(s)},\mbox{ and }\,\kappa=\eta\,\Vert\alpha''(s)\Vert_{z,{\mathrm{p}}}=-\epsilon\Vert\tilde{\alpha}''(s)\Vert_1,
\end{equation}
where $\eta=\langle\mathbf{n},\mathbf{n}\rangle_{z,{\mathrm{p}}}=-\epsilon$ (note that $\mathbf{t}'$ is not lightlike). Note that the curvature function is just the curvature of the top view curve $\tilde{\alpha}$ in $\mathbb{E}_1^2$. For the binormal we define $\mathbf{b}=(0,0,1)\Rightarrow \mathbf{b}'=0$. Clearly we have $\mathbf{t}'=\eta\kappa\,\mathbf{n}=-\epsilon\kappa\mathbf{n}$. For the derivative of the principal normal, let us write
$\mathbf{n}' = a\,\mathbf{t}+b\,\mathbf{n}+c\,\mathbf{b}$. Since $\langle\mathbf{n},\mathbf{n}\rangle_{z,{\mathrm{p}}}=\eta=\pm1$, we necessarily have $b=0$. On the other hand, for the first coefficient
$a=\epsilon\langle\mathbf{n}',\mathbf{t}\rangle_{z,{\mathrm{p}}}=-\epsilon\langle\mathbf{n},\mathbf{t}'\rangle_{z,{\mathrm{p}}}=-\epsilon\kappa$. Finally, from the third coefficient we define the pseudo-isotropic torsion $c=-\epsilon\eta\tau=\tau$, in analogy with the definition of torsion in $\mathbb{E}_1^3$ \cite{daSilvaJG2017,Kuhnel2010}. In short, we have the following pseudo-isotropic Frenet equations
\begin{equation}
\frac{\mathrm{d}}{\mathrm{d}s}\left(
\begin{array}{c}
\mathbf{t}\\
\mathbf{n}\\
\mathbf{b}\\
\end{array}
\right)=\left(
\begin{array}{ccc}
0 \, & -\epsilon\kappa \, & 0 \, \\
-\epsilon\kappa \, & 0 \, & \tau \, \\
0\, & 0 \,& 0\,\\
\end{array}
\right)\left(
\begin{array}{c}
\mathbf{t}\\
\mathbf{n}\\
\mathbf{b}\\
\end{array}
\right).
\end{equation}

An admissible curve $\alpha:I\to\mathbb{I}_{\mathrm{p}}^3$ is a plane curve if and only if $\tau=0$. Indeed, from a pseudo-Euclidean viewpoint, the osculating plane has a normal vector given by $
\mathbf{u}=\alpha'\times_1\alpha''/\Vert\alpha'\times_1\alpha''\Vert_1$. The condition of being a plane curve is equivalent to $\mathbf{u}'\equiv0$, i.e., the osculating planes are always the same. Taking the derivative of $\mathbf{u}$ and using that $\alpha'\times\alpha''=\eta\kappa\,\mathbf{t}\times_1\mathbf{n}$ in combination with Frenet equations gives
\begin{equation}
\frac{\mathrm{d}\mathbf{u}}{\mathrm{d}s}=\tau\,\mathbf{t}\times_1\displaystyle\left(\frac{\mathbf{b}}{\Vert\mathbf{t}\times_1\mathbf{n}\Vert_1}-\frac{\langle\mathbf{t}\times_1\mathbf{n},\mathbf{t}\times_1\mathbf{b}\rangle}{\Vert\mathbf{t}\times_1\mathbf{n}\Vert_1^3}\,\mathbf{n}\right)=0\Leftrightarrow\tau=0.
\end{equation}

\subsection{Pseudo-isotropic spherical image and moving bivectors}

Let $\alpha:I\to\mathbb{I}_{\mathrm{p}}^3$ be an admissible curve and $\Sigma_{{\mathrm{pi}}}$ be the unit radius, $p=1$, sphere $
z = \frac{1}{2}(x^2-y^2)$.

\begin{definition}
 For each $s$, let $\alpha^*(s)$ be the point on $\Sigma_{si}$ such that the tangent plane $\Pi_s$ to $\Sigma_{{\mathrm{pi}}}$ at $\alpha^*(s)$ is parallel to the osculating plane $\pi_s$ of $\alpha$ at $\alpha(s)$. The curve $\alpha^*$ is the \emph{spherical image} of $\alpha$ (in $\mathbb{E}^3$, there are three types of spherical images, or indicatrices: $\mathbf{t}:I\to\mathbb{S}^2$, $\mathbf{n}:I\to\mathbb{S}^2$, and $\mathbf{b}:I\to\mathbb{S}^2$. In $\mathbb{I}_{\mathrm{p}}^3$, however, one can define non-trivial indicatrices only for the tangent and normal and they are curves on the unit sphere of cylindrical type).
\end{definition}

The equation of the tangent plane to $\Sigma_{{\mathrm{pi}}}$ at $\alpha^*$ is $
z = x^*x-y^*y-z^*$. On the other hand, from Eq. (\ref{eq::SemiIsoOscPlane}), the equation for the osculating plane is
\begin{equation}
z = \frac{\epsilon}{\kappa}(y'z''-y''z')x-\frac{\epsilon}{\kappa}(x'z''-x''z')y+w,
\end{equation}
where we used that $\kappa=-\epsilon\,(x'y''-x''y')$: the value of $w$ is not important here.

The condition $\Pi_s\parallel\pi_s$ leads to
\begin{equation}
x^* = \frac{\epsilon}{\kappa}\left\vert
\begin{array}{cc}
y'  & z' \\
y'' & z''\\
\end{array}
\right\vert,\,\,y^* = \frac{\epsilon}{\kappa}\left\vert
\begin{array}{cc}
x'  & z' \\
x'' & z''\\
\end{array}
\right\vert.
\end{equation}
Finally, in order to find $z^*$, one may use that $x'\,^2-y'\,^2=\epsilon$ ($\Rightarrow x'x''-y'y''=0$) and $\kappa^2=\eta(x''^2-y''^2)$, $\eta=-\epsilon$. Then, 
\begin{equation}
z^* = \frac{1}{2}({x^*}^2-{y^*}^2)= \frac{\epsilon}{2\kappa^2}(\kappa^2z'^2-z''^2).
\end{equation}

The spherical image will be used to describe the pseudo-isotropic moving bivectors, defined by using the vector product in $\mathbb{E}_1^3$: $
\mathcal{T}=\mathbf{n}\times_1\mathbf{b}$, $\mathcal{N} = \mathbf{b}\times_1\mathbf{t}$, and $\mathcal{B} = \mathbf{t}\times_1\mathbf{n}$.

\begin{proposition}
The Frenet bivectors satisfy
\begin{equation}
\mathcal{B}=\mathbf{b}-\widetilde{\alpha^*},\,\mathcal{N}=\eta\,\tilde{\mathbf{n}},\,\mathcal{T}=\epsilon\,\tilde{\mathbf{t}}.
\end{equation}
\end{proposition}
\begin{proof}
We have $\mathcal{B}=\mathbf{t}\times_1\mathbf{n}=-\epsilon\alpha'\times_1\alpha''/\kappa$ and so
\begin{eqnarray}
\mathcal{B} & = & -\frac{\epsilon}{\kappa}(y'z''-y''z',x'z''-x''z',x'y''-x''y')\nonumber\\
& = &(-\epsilon\frac{y'z''-y''z'}{\kappa},-\epsilon\frac{x'z''-x''z'}{\kappa},1)=\mathbf{b}-\widetilde{\alpha^*}.
\end{eqnarray}
On the other hand, since $x'x''-y'y''=0$, we have $x''/(x'y''-x''y')=\epsilon y'$ and $y''/(x'y''-x''y')=\epsilon x'$. So, $\mathcal{T}=-\epsilon\alpha''\times_1\mathbf{k}/\kappa=\epsilon\tilde{\mathbf{t}}$. Similarly, $\mathcal{N}=-\epsilon\, \tilde{\mathbf{n}}$.
\end{proof}

\subsection{Rotation minimizing frames in pseudo-isotropic space}

As in $\mathbb{I}^3$, the binormal $\mathbf{b}$ is RM: $\mathbf{b}'=0$. Thus, we need to introduce an RM vector field in substitution to the principal normal $\mathbf{n}$. As in $\mathbb{I}^3$, we have the 
\begin{theorem}
Let $\mathbf{n}_1$ be a unit normal vector field along $\alpha:I\to\mathbb{I}_{\mathrm{p}}^3$. If $\mathbf{n}_1$ is RM, then
\begin{equation}
 \mathbf{n}_1(s)=\mathbf{n}(s)-\left(\int_{s_0}^s \tau(x)\mathrm{d}x+\tau_0\right)\,\mathbf{b}(s),
\end{equation}
where $\tau_0$ is a constant. In addition, an RM frame $\{\mathbf{t},\mathbf{n}_1,\mathbf{n}_2=\mathbf{b}\}$ in pseudo-isotropic space $\mathbb{I}_{\mathrm{p}}^3$ satisfies\footnote{It may be instructive to compare this equation of motion with Eq. (16) from \cite{daSilvaJG2017}.}
\begin{equation}
\frac{\mathrm{d}}{\mathrm{d}s}\left(
\begin{array}{c}
\mathbf{t}\\
\mathbf{n}_1\\
\mathbf{n}_2\\
\end{array}
\right)=\left(
\begin{array}{ccc}
0 & \eta\kappa_1 & \kappa_2\\
-\epsilon\kappa_1 & 0 & 0\\
0 & 0 & 0\\
\end{array}
\right)\left(
\begin{array}{c}
\mathbf{t}\\
\mathbf{n}_1\\
\mathbf{n}_2\\
\end{array}
\right),
\end{equation}
where the natural curvatures are $\kappa_1=\kappa=\kappa\,\cosg\,\theta$ and $\kappa_2=\eta\,\kappa\,\theta=-\epsilon\,\kappa\,\sing\,\theta$, with $\sing\,[\theta(s)]=\theta(s)=\int_{s_0}^s \tau(x)\mathrm{d}x+\tau_0$.
\end{theorem}

We can also introduce  moving bivector frame associated with an RM frame
\begin{equation}
\left\{
\begin{array}{ccc}
\mathcal{T} &=& \mathbf{n}_1\times_1\mathbf{b}\\
\mathcal{N}_1 &=& \mathbf{b}\times_1\mathbf{t}\\
\mathcal{N}_2 &=& \mathbf{t}\times_1\mathbf{n}_1\\
\end{array}
\right.\Rightarrow
\frac{\mathrm{d}}{\mathrm{d}s}\left(
\begin{array}{c}
\mathcal{T}\\
\mathcal{N}_1\\
\mathcal{N}_2\\
\end{array}
\right)=\left(
\begin{array}{ccc}
0 & \epsilon\kappa_1 & 0\\
-\eta\kappa_1 & 0 & 0\\
-\kappa_2 & 0 & 0\\
\end{array}
\right)\left(
\begin{array}{c}
\mathcal{T}\\
\mathcal{N}_1\\
\mathcal{N}_2\\
\end{array}
\right).
\end{equation}

\section{pseudo-isotropic spherical curves}

All the results to be described below are analogous to their simply isotropic versions and, therefore, we will not go through the details. 

We may write a pseudo-isotropic osculating sphere at $\alpha_0=\alpha(s_0)$ as
\begin{equation}
F_1(\mathbf{x})=\lambda\langle\mathbf{x}-\alpha_0,\mathbf{x}-\alpha_0\rangle_{z,s}+\langle\mathbf{u},\mathbf{x}-\alpha_0\rangle_1=0,
\end{equation}
where the constants $\lambda\in\mathbb{R},\mathbf{u}\in\mathbb{E}_1^3$ will be determined by the contact of order 3 condition. Taking the derivatives $(F\circ\alpha)^{(k)}(s)$, $k=1,2,3$, at $s=s_0$ gives
\begin{equation}
\langle\mathbf{u},\mathbf{t}\rangle_1=0,\,
2\lambda\epsilon+\langle\mathbf{u},\eta\kappa_1\mathbf{n}_1+\kappa_2\mathbf{n}_2\rangle_1=0,\,\mbox{ and }\,
\langle\mathbf{u},\eta\kappa_1'\mathbf{n}_1+\kappa_2'\mathbf{n}_2\rangle_1=0.
\end{equation}
By using the RM moving bivectors, we deduce that for some constant $\rho\not=0$
\begin{equation}
\mathbf{u}=\rho[\eta\kappa_1'\mathcal{N}_2-\kappa_2'\mathcal{N}_1]\vert_{s_0},\,\lambda\,\epsilon = \rho\,\tau\kappa^2.
\end{equation}

In short, the equation for a pseudo-isotropic osculating sphere can be written as
\begin{equation}
\tilde{\mathbf{x}}^2-2\langle\mathbf{x},\tilde{\alpha}_0+\frac{\kappa_2'\mathcal{N}_1-\eta\kappa_1'\mathcal{N}_2}{\epsilon\tau\kappa^2}\vert_{s_0}\rangle_1+2\left[\frac{\tilde{\alpha}_0^2}{2}-\langle\alpha_0,\frac{\eta\kappa_1'\mathcal{N}_2-\kappa_2'\mathcal{N}_1}{\epsilon\tau\kappa^2}\vert_{s_0}\rangle_1\right]=0.
\end{equation}

The condition of being spherical implies that the pseudo-isotropic osculating spheres are all the same. This condition demands
\begin{equation}
\frac{\mathrm{d}}{\mathrm{d}s}\left[\tilde{\alpha}+\frac{\kappa_2'\mathcal{N}_1}{\epsilon\tau\kappa^2}-\frac{\eta\kappa_1'\mathcal{N}_2}{\epsilon\tau\kappa^2}\right]=0,\,
\frac{\mathrm{d}}{\mathrm{d}s}\left[\frac{\tilde{\alpha}^2}{2}-\left\langle\alpha,\frac{\eta\kappa_1'\mathcal{N}_2-\kappa_2'\mathcal{N}_1}{\epsilon\tau\kappa^2}\right\rangle\right]=0.\label{eq::CondToBeSemiSphericalScalar}
\end{equation}

The first equation above leads to
$
a_1:=\frac{\kappa_2'}{\tau\kappa^2},$ and $\,a_2:=-\frac{\kappa_1'}{\tau\kappa^2}
$ constant, while the second gives $
(\epsilon+a_1\kappa_1+a_2\kappa_2)\langle\alpha,\mathbf{t}\rangle_{z,s}=0\,.
$ Then, we have the

\begin{theorem}
An admissible $C^4$ regular curve $\alpha:I\to \mathbb{I}_{\mathrm{p}}^3$ lies on the surface of a sphere if and only if its normal development $(\kappa_1(s),\kappa_2(s))$ lies on a line not passing through the origin. In addition, $\alpha$ is a plane curve if and only if the normal development lies on a line passing through the origin. 
\end{theorem}

\appendix

\section{Generalized complex numbers}

\begin{figure*}[tbp]
\centering
    {\includegraphics[width=0.4\linewidth]{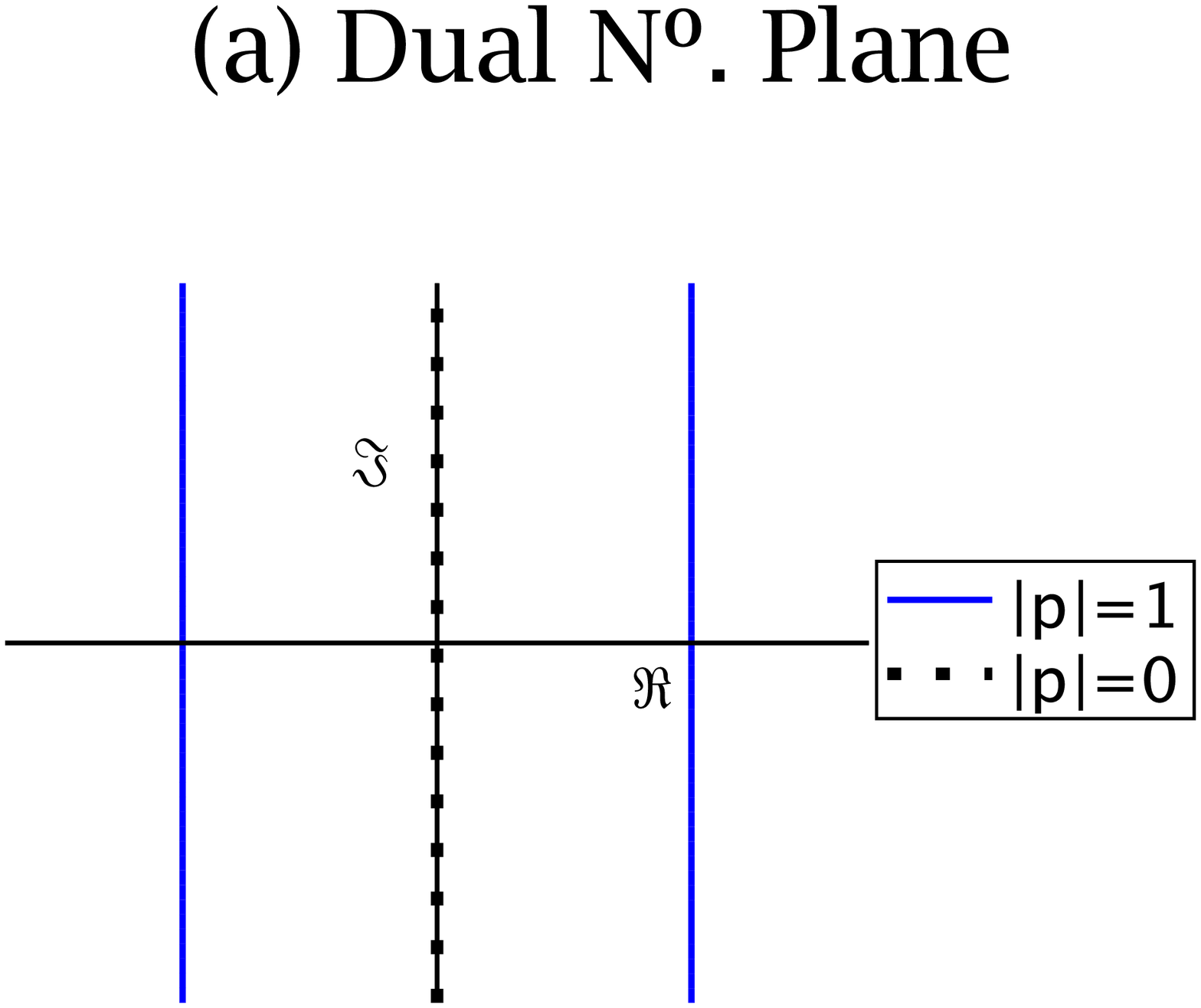}}
    {\includegraphics[width=0.4\linewidth]{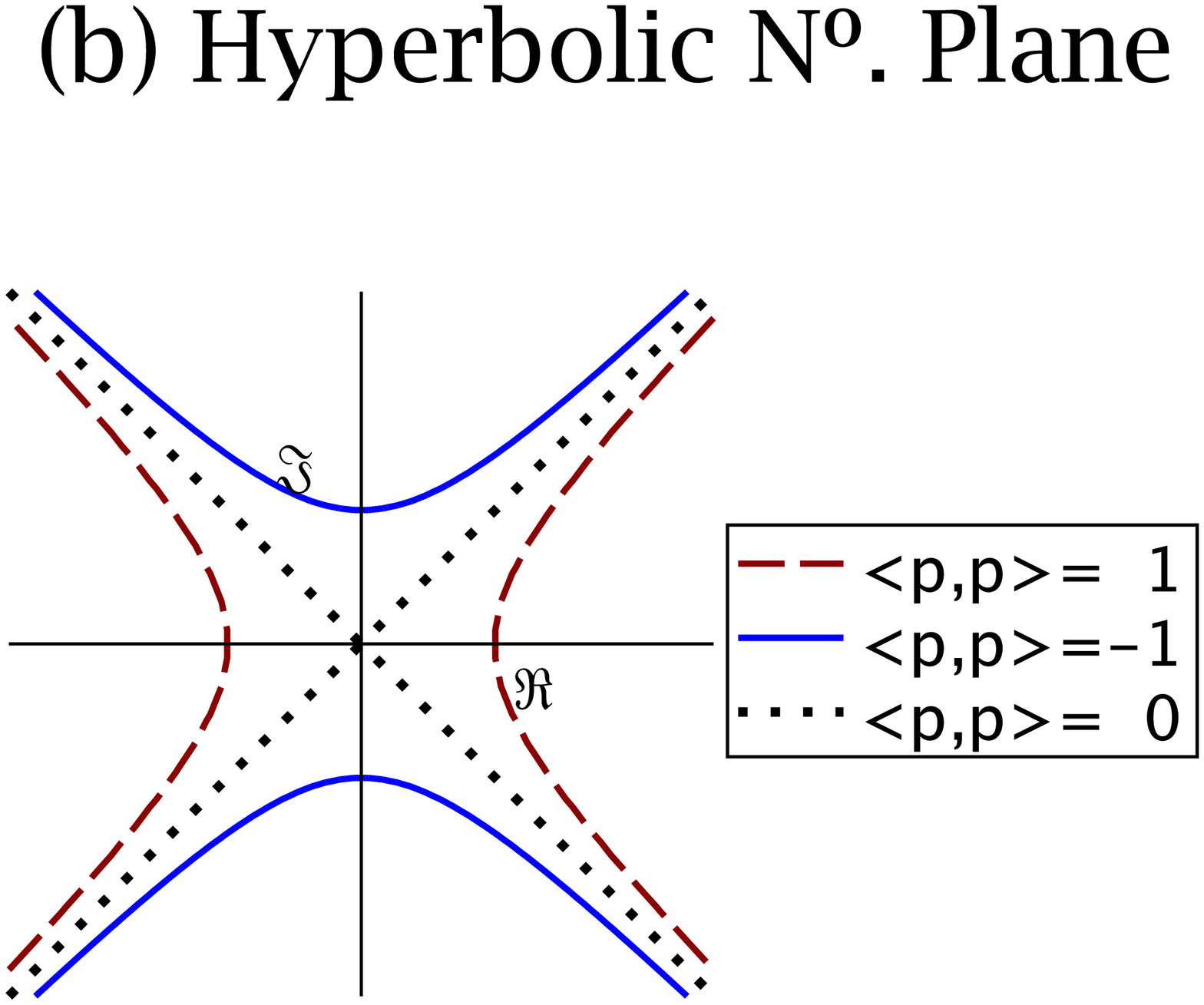}}
   
          \caption{Unit length and zero divisors hyperbolic and dual numbers: (a) Dual numbers zero divisors (dotted black line), and unit length dual numbers (solid blue line). A unit length dual number may be written as $(\pm1,\phi)\mapsto \pm1+\phi\,\varepsilon$; (b) Hyperbolic numbers zero divisors (dotted black line), and unit length hyperbolic numbers (solid blue and dashed red lines). There exist four types of unit length hyperbolic numbers, the dashed red lines correspond to $(\pm\cosh\phi,\sinh\phi)\mapsto\pm\cosh\phi+\ell\sinh\phi$ and the solid blue lines correspond to $(\sinh\phi,\pm\cosh\phi)\mapsto\sinh\phi\pm\ell\cosh\phi$. In particular, the dashed red line for $\Re>0$ corresponds to $O^{++}(2)$, the group of time orientation preserving hyperbolic rotations (see Subsect. \ref{subsec::DescSemiIsoGeom}).}
          \label{fig::DiagramSphPlaneCurv}
\end{figure*}

In addition to the well known (field of) complex numbers $\mathbb{C}$, we may also extend the reals to numbers in a plane by specifying other values for the square of the imaginary unity: dual numbers for a vanishing square and Lorentz numbers for a positive square, as described below.

\subsection{The ring of dual numbers}

We write a \emph{dual number} $p\in\mathbb{D}$ as $p=p_1+p_2\,\varepsilon$, where the dual imaginary $\varepsilon$ satisfies $\varepsilon^2=0$ \cite{Sachs1987,Yaglom1979}. Algebraically, the ring $\mathbb{D}$ is isomorphic to $\mathbb{R}[X]/(X^2)$. The real and imaginary parts are $p_1=\Re(p)$ and $p_2=\Im(p)$, respectively. The arithmetic operations in $\mathbb{D}$ are defined as 
\begin{equation}
\left\{
\begin{array}{rcl}
(p_1+p_2\,\varepsilon)+(q_1+q_2\,\varepsilon) & = & (p_1+q_1)+(p_2+q_2)\,\varepsilon\\ 
(p_1+p_2\,\varepsilon)(q_1+q_2\,\varepsilon) & = & (p_1q_1)+(p_1q_2+q_1p_2)\,\varepsilon
\end{array}
\right..
\end{equation}
In addition, we may introduce a (semi-)norm in  $\mathbb{I}^2$ by $\vert p_1+p_2\varepsilon\vert=\vert p_1\vert$, which is induced by $\mathbf{u}\,\cdot_{\,\mathbb{I}^2}\,\mathbf{v}=u_1v_1$.  Unit duals can be written as $p=1+\phi\,\varepsilon=\cosg\phi+\varepsilon\,\sing\phi$, where $\cosg\phi$ and $\sing\phi$ are the Galilean trigonometric functions \cite{Yaglom1979}.

Finally, dual numbers $\mathbb{D}$ admit a linear representation in the 2 by 2 matrices 
\begin{equation}
p_1+p_2\,\varepsilon\mapsto\left(
\begin{array}{cc}
p_1 & 0\\
p_2 & p_1\\
\end{array}
\right).\label{eq::LinearRepDualNumbers}
\end{equation}

\subsection{The ring of Lorentz/hyperbolic numbers}

We write a Lorentz number as $p=p_1+p_2\ell$, where the hyperbolic imaginary $\ell$ satisfies $\ell^2=1$ \cite{BirmanMonthly1984,Yaglom1979}. Algebraically, the ring $\mathbb{L}$ is isomorphic to $\mathbb{R}[X]/(X^2-1)$. The real and imaginary parts are $p_1=\Re(p)$ and $p_2=\Im(p)$, respectively. The arithmetic operations in $\mathbb{L}$ are 
\begin{equation}
\left\{
\begin{array}{rcl}
(p_1+p_2\ell)+(q_1+q_2\ell) &=& (p_1+q_1)+(p_2+q_2)\ell\\
(p_1+p_2\ell)(q_1+q_2\ell) &=& (p_1q_1+p_2q_2)+(p_1q_2+q_1p_2)\ell
\end{array}
\right..
\end{equation} 

An inner product in $\mathbb{E}_1^2$ may be introduced as $\langle p,q\rangle_1=\Re(p\,\bar{q})$, where $\bar{p}=\overline{p_1+p_2\ell}=p_1-p_2\ell$ denotes hyperbolic conjugation (see Fig. \ref{fig::DiagramSphPlaneCurv}). Unit Lorentz numbers can be written as $p=(\cosh\theta+\ell\sinh\theta)$ if $p$ is spacelike, i.e., $ p\,\bar{p}>0$, or as $p=(\sinh\theta+\ell\cosh\theta)$ if $p$ is timelike, i.e., $ p\,\bar{p}<0$. Notice that Lorentz numbers admit a linear representation in the 2 by 2 matrices as
\begin{equation}
p_1+p_2\,\ell\mapsto\left(
\begin{array}{cc}
p_1 & p_2\\
p_2 & p_1\\
\end{array}
\right).\label{eq::LinearRepLorentzNumbersNonNullBasis}
\end{equation}

On the other hand, there is a description of $\mathbb{L}$ distinct from the (canonical) basis $\{1,\ell\}$. Indeed, we can describe a hyperbolic number $p$ in terms of the light-cone basis $\mathrm{e}_{\pm}=(1\pm\ell)/2$. The number $\mathrm{e}_{\pm}$ is lightlike, i.e., $\mathrm{e}_{\pm}\,\overline{\mathrm{e}_{\pm}}=0$, and writing $p=p_+\,\mathrm{e}_+ + p_-\,\mathrm{e}_-$, we have 
\begin{equation}
\left\{
\begin{array}{ccl}
p\,q &=& (p_{+}q_{+})\,\mathrm{e}_+\,+\,(p_-q_-)\,\mathrm{e}_-\\
\bar{p} &=& p_{-}\mathrm{e}_{+}+p_{+}\mathrm{e}_{-}\\
\mathrm{e}_{\pm}^2 &=& \mathrm{e}_{\pm}\,,\,\,\mathrm{e}_{\pm}\mathrm{e}_{\mp}=0
\end{array}
\right..
\end{equation}
In the light-cone basis, the linear representation in the 2 by 2 matrices reads
\begin{equation}
p_+\,\mathrm{e}_+ + p_-\,\mathrm{e}_-\mapsto\left(
\begin{array}{cc}
p_+ & 0\\
0 & p_-\\
\end{array}
\right).\label{eq::LinearRepLorentzNumbersNullBasis}
\end{equation}




\end{document}